\documentclass[12pt]{amsart} 
\usepackage{latexsym,enumerate,amscd,amssymb}
 \usepackage{hyperref}
\usepackage{fullpage}
\usepackage[nocolor]{sseq}
\usepackage[textwidth=1.2in, textsize=small]{todonotes}
\usepackage{tikz, tikz-cd}
\usetikzlibrary{arrows}
\usetikzlibrary{matrix}
\usetikzlibrary{shapes}
\usetikzlibrary{calc}
\usepgflibrary{shapes.geometric}
\usepackage{thmtools} 
\usepackage{thm-restate}
\usepackage{mathabx}

\theoremstyle{definition} 
\newtheorem{Theorem}{Theorem}[section]

\newtheorem{Corollary}[Theorem]{Corollary} 
\newtheorem{Lemma}[Theorem]{Lemma}
\newtheorem{Proposition}[Theorem]{Proposition}
\newtheorem{Definition}[Theorem]{Definition} 

\newcommand{\llangle}{\langle\!\langle}
\newcommand{\rrangle}{\rangle\!\rangle}

\theoremstyle{plain} 
\newtheorem{step}{Step}

\newtheorem*{Strong Atiyah conjecture}{Strong Atiyah conjecture}
\newtheorem*{Strong Atiyah conjecture(Algebraic Version)}{Strong Atiyah conjecture (Algebraic Version)}

\newcommand{\inv}[1]{{#1}^{-1}}

\DeclareMathOperator{\Id}{Id}

\DeclareMathOperator{\Aut}{Aut}

\title{Profinite properties of algebraically clean graphs of free groups}
\author{Kasia Jankiewicz}
\author{Kevin Schreve}
\address{Department of Mathematics, University of California, Santa Cruz, CA 95064}
\email{kasia@ucsc.edu}
\address{Department of Mathematics, Louisiana State University, Baton Rouge, LA 70806}
\email{kschreve@lsu.edu}
\makeatletter
\@namedef{subjclassname@2020}{\textup{2020} Mathematics Subject Classification}
\makeatother


\begin{document}

\maketitle

\begin{abstract}
We prove that for every prime $p$ algebraically clean graphs of groups are virtually residually $p$-finite and cohomologically $p$-complete. We also prove that they are cohomologically good. We apply this to certain $2$-dimensional Artin groups.

\end{abstract}

\section{Introduction}

An \emph{algebraically clean graph of free groups} is a graph of groups where each vertex group and edge group are finite rank free groups, and all the inclusion maps are inclusions of free factors. Examples of the fundamental groups of algebraically clean graph groups include free-by-cyclic groups, the fundamental groups of clean $2$-complexes in the sense of Wise \cite{WiseCycHNN}, and certain $2$-dimensional Artin groups \cite{JankiewiczArtinRf, JankiewiczArtinSplittings}. We note that (many among) the former examples are known to not admit (virtual) cocompact actions on CAT(0) cube complexes, so they are not virtually cocompactly special. In particular, the family of algebraically clean graphs of groups is strictly larger than the family of the fundamental groups of finite clean $2$-complexes, which all are virtually special.

\subsection{Virtual residual $p$-finiteness}
A group $G$ is \emph{residually finite} if for every $g\in G-\{1\}$ there exists a quotient $\phi:G\to K$ where $K$ is a finite group and $\phi(g) \neq 1$. The fundamental groups of algebraically clean graph of free groups are known to be residually finite \cite[Thm 3.4]{WisePolygons}.

Let $p$ be a prime number. 
A group $G$ is \emph{residually $p$-finite} if for every $g\in G-\{1\}$ there exists a quotient $\phi:G\to K$ where $K$ is a finite $p$-group and $\phi(g) \neq 1$. Clearly, every residually $p$-finite group is residually finite, but the converse does not hold. 

\begin{Theorem}\label{thm: main residually p}
For every prime $p$, the fundamental group of an algebraically clean graph of free groups has a finite index subgroup that is residually $p$-finite. 
\end{Theorem}

We do not know whether algebraically clean graphs of free groups are linear. We note that linear groups are known to be virtually residually $p$-finite \cite{Platonov68} for all but finitely many primes $p$.
There have been previous combination theorems concerning residual $p$-finiteness, originating in the work of Higman \cite{Higman64}, see e.g.\ \cite{Wilkes2019, Azarov2017, Sokolov2023} and references therein.

\subsection{Cohomology of profinite and pro-$p$ completions}
For a residually finite group $G$, the \emph{profinite completion} $\widehat G$ of $G$ is defined as 
$$\widehat G = \varprojlim_{[G:H] < \infty} G/H,$$
where the inverse limit is taken over the system of finite quotients of $G$. 
For every $G$, there is a canonical homomorphism $i: G \rightarrow \widehat G$ which sends $g \in G$ to the cosets $gH$.
 A  group $G$ is called \textit{cohomologically good} (also known as \emph{good in the sense of Serre}) if for every finite $G$-module $M$ the induced homomorphism $$H^\ast_{cont}(\widehat G, M) = \varprojlim_{[G:H] < \infty} H^\ast(G/H, M) \overset{i^\ast}\rightarrow H^\ast(G, M)$$ is an isomorphism. We always take the cohomology of a profinite group to be its continuous cohomology. 
Goodness was introduced in \cite[Exercises 2.6]{SerreGalois97}

We can analogously define \emph{cohomological $p$-completeness} for a residually $p$-finite group. In this case, the \emph{pro-$p$ completion} $\widehat G_p$ of $G$ is given by
\[
\widehat G_p = \varprojlim G/H
\]
where $H$ varies over all the subgroups of $G$ whose index is a power of $p$. 
Then $G$ is \emph{cohomologically $p$-complete} if the homomorphism $G\to\widehat G_p$ induces an isomorphism \[H^*_{cont}(\widehat G_p, \mathbb F_p)\to H^*(G, \mathbb F_p)\] 
where we assume the $G$-action on $\mathbb{F}_p$ is trivial.

\begin{Theorem}\label{thm: main good}
The fundamental group of an algebraically clean graph of free groups is 
\begin{enumerate}
\item cohomologically good,
\item for every prime $p$, virtually cohomologically $p$-complete.
\end{enumerate}
\end{Theorem}

For each $p$, the cohomologically complete finite index subgroup is a priori different. 
General graphs of free groups do not always satisfy the above theorem. Indeed, there exist examples of amalgamated products of free groups that are not residually finite \cite{Bhattacharjee94, Wise96Thesis}. There are even examples of simple groups that split as amalgamated products of free groups \cite{BurgerMozes97}.

\subsection{Virtual poly-freeness}
A group $G$ is \emph{poly-free} if it admits a chain of subgroups $1 = G_0\unlhd G_1\unlhd \dots\unlhd G_n =G$ such that $G_{i}/G_{G_i-1}$ is a free group (of possibly infinite rank). We say $G$ is \emph{normally poly-free} if additionally all subgroups $G_i$ are normal in $G$. 

\begin{Theorem}\label{thm:poly-free}
Algebraically clean graphs of free groups are normally poly-free.
\end{Theorem}

This has a number of consequences; for instance it implies these groups are locally indicable, hence left-orderable \cite{RhemtullaRolfsen2002}, and satisfy the $K-$ and $L-$theoretic Farrell-Jones Conjecture \cite{bestvina2021farrelljones, BruckKielakWu21}.


\subsection{Applications to Artin groups}
An \emph{Artin group} is given by a presentation
$$A = \langle s_1,\dots, s_k | \underbrace{s_is_js_i\cdots}_{m_{ij} \hspace{.5mm} terms} = \underbrace{s_js_is_j\cdots}_{m_{ij} \hspace{.5mm} terms} \rangle$$ 
where $m_{ij} \in \{2,3,\dots\} \cup \{\infty\}$. We understand $m_{ij} = \infty$ as no relation involving $s_i$ and $s_j$.
A \emph{triangle Artin group} $A_{\ell mn}$ is an Artin group where $k=3$, and $m_{12} = \ell$, $m_{23} = m$, and $m_{13}= n$.

 Since finite type Artin groups are linear (\cite{Krammer2002}, \cite{Bigelow2001} for braid groups, and \cite{CohenWales2002}, \cite{Digne2003} in general), it follows that they are also virtually residually $p$-finite. Moreover, pure Artin groups of type $A_n, C_n, G_2$ and $I_2(n)$ 
are residually $p$ and cohomologically $p$-complete for all $p$ \cite{AschenbrennerFriedl13}, and cohomologically good \cite{SerreGalois97}.

With the next corollary in mind, we note that the only spherical triangle Artin groups are the $A_{22n} = A(I_2(n))\times \mathbb Z$ for $n\geq 2$, and $A_{23n}$ where $n\in\{3,4,5\}$. Among those, the even ones, $A_{22n}$ for even $n$, all are known be cohomologically good and virtually residually $p$-finite and cohomologically $p$-complete for all $p$.

\begin{Corollary}\label{cor:Artin}
A triangle Artin $A_{\ell mn}$ where $\ell\leq m\leq n$ is 
\begin{itemize}
\item residually finite and cohomologically good,
\item for each prime $p$, virtually residually $p$-finite and cohomologically $p$-complete, 
\item virtually normally poly-free,
\end{itemize}
provided that 
\begin{itemize}
\item $\ell = 2$, and $m,n\geq 4$ and at least one of them is even, or
\item $\ell,m,n\geq 4$ except for the case where $\ell = m = 4$ and $n$ is odd, or
\end{itemize}
In particular, all even triangle Artin groups and all extra-extra-large triangle Artin groups (i.e.\ where $\ell,m,n\geq 5$) satisfy the above.

Moreover, there are many more $2$-dimensional Artin groups that have the above properties. See \cite{JankiewiczArtinRf} for a combinatorial criterion on the defining graph, which ensure that the associated Artin group is virtually algebraically clean graph of free groups.
\end{Corollary}
The Artin groups above were shown to virtually split as algebraically clean graphs of free groups in \cite{JankiewiczArtinRf, JankiewiczArtinSplittings}. We note that ``virtual'' in the above statement is necessary. Indeed, a group $G$ that is residually $p$-finite for all  primes $p$ is bi-orderable \cite{Rhemtulla73} (see also \cite{KoberdaSuciu20}), but the only bi-orderable Artin groups are right-angled Artin groups.  However, it is possible that each Artin groups listed above contains a finite index subgroup that is residually $p$-finite for all primes $p$.

Artin groups that are known to be poly-free are right-angled Artin groups \cite{DuchampKrob93, Howie99, HermillerSunic07}, even FC-type Artin groups \cite{BlascoGarciaMartinezPerezParis19}, and even large type Artin groups \cite{BlascoGarcia21}. Artin groups of types $A_n$, $B_n = C_n$, $D_n$, $F_4$, $G_2$ and $I_2(n)$ \cite{Brieskron73}, as well as $\widetilde A_n$, $\widetilde B_n$, $\widetilde C_n$, $\widetilde D_n$ \cite{Roushon2020} are known to be virtually poly-free.
Independently, Wu-Ye proved that all triangle Artin groups except $A_{23n}$ where $n$ is odd, are virtually poly-free \cite{WuYe2023}. Wu-Ye also show that some triangle Artin groups are not poly-free. \\

Finally, we also establish residual finiteness and cohomological goodness for all even Artin groups whose defining graphs contains no $4$-cliques.
\begin{Theorem}\label{thm:evenArtin}
Let $\Gamma$ be a finite labelled graph with all even labels that does not contain a $4$-clique. Then $A_{\Gamma}$ is residually finite and cohomologically good.
\end{Theorem}
Such Artin groups are also poly-free by \cite{BlascoGarcia21, Wu22}.

\subsection*{Acknowledgements}
The first author was supported by the NSF grant DMS-2203307 and DMS-2238198. The second author was supported by the NSF grant DMS-2203325.

\section{Graphs of groups}
\subsection{Graph of groups notation} 
We recall the basic definitions and set the notation. 

A \emph{graph} $Y$ consists of a set $V(Y)$ of its \emph{vertices}, and a set $E(Y)$ of its \emph{edges}, and two maps:
\begin{enumerate}
\item $\overline{\cdot}:E(Y)\to E(Y)$ satisfying $\overline{\overline e} = e$, where we think of $\overline e$ as the edge $e$ with the orientation reversed,
\item $\tau: E(Y)\to V(Y)$, which we think of as taking the endpoint of an edge.
\end{enumerate}
A \emph{graph of groups} $\mathcal G$ with underlying graph $Y$ consists of a family of vertex groups $\{G_v\}_{v\in V(Y)}$ and edge groups $\{G_e\}_{e\in E(Y)}$ where $G_e = G_{\overline e}$ together with  maps $\{f_{e}:G_e\to G_{\tau(e)}\}_{e\in E(Y)}$.

Let $T\subseteq E(Y)$ be a set of edges of a spanning tree of $Y$.
The \emph{fundamental group} $\pi_1\mathcal G$ of the graph of groups $\mathcal G$ is constructed as the quotient 
\[
\pi_1\mathcal G = (\Asterisk_{v\in V(Y)} G_v  F(E(Y))/K
\]
where $K$ is a set of the following relations
\begin{enumerate}
\item $ef_e(g)\overline e = f_{\overline e}(g)$ for all $e\in E(Y)$ and $g \in G_e$, and
\item $\overline e = \inv{e}$, and $e = 1$ if and only if $e\in T$.
\end{enumerate}

\subsection{Algebraically clean graph of groups}
An \emph{algebraically clean graph of free groups} is a graph of groups $\mathcal G$ with finite underlying graph $Y$, where $G_v$ is a finite rank free group for all $v\in V(Y)$, $G_e$ is finite for all $e\in E(Y)$, and the maps $f_e:G_e\to G_{\tau(e)}$ are injective maps onto free factors.

Let $G$ be a group, and $N,M\subseteq G$ be two subgroups. We say that an isomorphism $\phi:N\to M$ is a \emph{partial automorphism}, if there exists an automorphism $\phi^{ext}:G\to G$ such that $\phi^{ext}_{|N} = \phi$. A \emph{partial identity} is a partial automorphism that can be extended to the identity. 

\begin{Proposition}\label{prop:multiple HNN}
Every algebraically clean graph of free groups $\mathcal G$ admits a splitting as an algebraically clean graph of groups $\mathcal G'$ where the underlying graph $Y'$ has a unique vertex, and up to renaming $e$ and $\overline e$, 
 $G_e \subseteq G_{\tau(e)}$ is a free factor, $f_e$ is the inclusion map, and $f_{\overline e}$ is a partial automorphism of $G_{\tau(e)}$.

\end{Proposition}
\begin{proof} Consider a spanning tree $T$ in the underlying graph $Y$ of $\mathcal G$. 
We define a new graph $Y'$ to have the vertex set $V(Y') = \{T\}$ and edge set $E(Y') = \{e\in E(Y)\mid e \notin E(T)\}$.
Then by ``collapsing'' $T$ in $Y$, we can identify $\pi_1 \mathcal G$ with the fundamental group of a graph of groups $\mathcal G'$ with underlying graph $Y'$, where 
\begin{itemize}
\item $G_T = \Asterisk_{v\in T}G_v$,
\item for each edge $e\notin E(T)$, $G_e$ becomes identified with $f_e(G_e)\subseteq G_{\tau(e)}\subseteq G_T$ which is a free factor in $G_{\tau(e)}$ and therefore also in $G_T$, and the map $f_{\overline e}:G_{\overline{e}} = G_e\to G_{\tau(\overline e)}\subseteq G_T$ is an embedding onto some free factor of $G_{\tau(\overline e)}$ and again also a free factor of $G_T$. We can thus think of that map $f_{\overline e}$ as a partial automorphism of $G_T$.
\end{itemize}
\end{proof}

\section{Residual $p$-finiteness}
Throughout this section $p$ is a fixed prime.

\subsection{Well-known basics on residual $p$-finiteness}
We start with stating some easy facts that we will use later.
\begin{Lemma}\label{lem: res p}
\item 
\begin{enumerate}
\item Let $N\unlhd G$ be a subgroup whose index is a power of $p$. Then there exists a characteristic subgroup $K\unlhd G$ whose index is a power of $p$, such that $K\subseteq N$. 
\item Let $G$ fit in a short exact sequence
$$1 \to N\to G \to Q \to 1$$
where $Q$ is a finite $p$-group, and $N$ is residually $p$-finite. Then $G$ is residually $p$-finite.
\end{enumerate}
\end{Lemma}
\begin{proof}
\item 
\begin{enumerate}
\item 
Let $K$ be the intersection $\bigcap H$ of all the normal subgroups $H$ of $G$ of index $[G:N]$. 
Note that $K$ is also the kernel of a homomorphism $G\to \prod_{H} G/H$, since the order of each $G/H$ is a power of $p$, so is the order of $\prod_{H} G/H$. In particular, the index $[G:K]$ is a power of $p$.
\item Let $g\in G$. If $g$ survives in $Q$, then $Q$ is the required finite $p$-quotient of $G$.
Suppose $g\in N$. Since $N$ is residually $p$-finite, then using (1) we know that there exists a characteristic subgroup $K\subseteq N$ such that $g\notin K$ and whose index is a power of $p$. Since $K$ is characteristic in $N$, it is normal in $G$, and $[G:K]$ is a power of $p$. 
\end{enumerate}
\end{proof}

\subsection{Basics on lower central $p$-series}
Let $G$ be a finitely generated group. For subgroup $H,K\subseteq G$ we denote:
\begin{itemize}
\item $H^p = \langle h^p\mid h\in H\rangle$, 
\item $[H, K] =  \langle [h,k]\mid h\in H, k \in K\rangle$, and we use the convention that $[h,k] =hkh^{-1}k^{-1}$, 
\item $HK = \langle hk\mid h\in H, k\in K\rangle$.
\end{itemize}

Let $G$ be a finitely generated group. A \emph{filtration} of $G$ is a collection $(G_n)_{n\in \mathbb N}$ of subgroups of $G$ where $G_1 = G$, and $G_{n+1}\subseteq G_n$ for each $n\in \mathbb N$. A filtration $(G_n)_{n\in \mathbb N}$ is \emph{normal} if $G_n\unlhd G$ is normal for each $n\in \mathbb N$, and it is \emph{separating} if $\bigcap_{n\in \mathbb N} G_n = \{1\}$.

The \emph{lower $p$-central filtration} $\{\gamma^p_n(G)\}_n$ of $G$ is defined as:
\[
\gamma^p_1(G) := G, \qquad \gamma^p_{n+1}(G) := \left(\gamma^p_{n}(G)\right)^p[G, \gamma^p_{n}(G)].
\] 

We also denote $L_n^p(G) = \gamma^p_{n}(G)/\gamma^p_{n+1}(G)$. In particular, $L_1^p(G) = H_1(G, \mathbb F_p)$.  The lower $p$-central filtration of $G$ is a normal filtration, and it is separating if and only if $G$ is residually $p$-finite.
 We note a couple of basic well-known properties of the lower $p$-central series. For completeness, we provide proofs.

\begin{Lemma}\label{prop: lower p-central filtration basic}
\item
\begin{enumerate}
\item For each $n$ we have $(\gamma_n^p(G))^{p}\subseteq \gamma_{n+1}^p(G)$.
\item For each $n,m$ we have $[\gamma_m^p(G),\gamma_n^p(G)]\subseteq \gamma_{n+m}^p(G)$. 
\item Each $\gamma_n^p(G)$ is a characteristic subgroup of $G$. In particular, for each $i$ there are natural homomorphisms 
$\theta_n:\Aut(G)\to\Aut(L_n^p(G))$ and $\sigma_n:\Aut(G)\to\Aut(G/\gamma^p_n(G))$.
\end{enumerate}
\end{Lemma}

\begin{proof}\item
\begin{enumerate}
\item Follows immediately from the definition.
\item We induct on $m$. For $m = 1$ the statement follows directly from the definition for every $n$. Suppose that $[\gamma_{m-1}^p(G),\gamma_n^p(G)]\subseteq \gamma_{n+m-1}^p(G)$ for every $n$. 


First we claim that $[(\gamma_{m-1}^p(G))^p, \gamma_n^p(G)]\subseteq \gamma_{n+m}^p(G)$.
Given $k\in \gamma_{n}^p(G)$ and $h\in\gamma_{m-1}^p(G)$ we need to show that $[h^p, k]\in\gamma_{n+m}^p(G)$. First note that $[h^p, k] = h^pu^p$ where $u = k\inv{h}\inv{k}$.
By the inductive assumption $u = \inv{h}\ell$ for some $\ell\in\gamma_{n+m-1}^p(F)$.

We have
\begin{align*}
[h, k]^p &= (hk\inv{h}\inv{k})^p\\
& = h^p(h^{-(p-1)}uh^{(p-1)})(h^{-(p-2)}uh^{(p-2)})\dots (\inv{h}uh)u.
\end{align*}
By substituting $u = \inv{h}\ell$ we get
\begin{align*}
h^{-(p-i)}uh^{(p-i)} = h^{-(p-i)}\inv{h}\ell h^{(p-i)} = \inv{h} h^{-(p-i)}\ell h^{(p-i)} =  \inv{h}\ell\ell_{i} = u\ell_i
\end{align*}
for some $\ell_{i}\in \gamma_{n+m}$ where the equality $h^{-(p-i)}\ell h^{(p-i)} = \ell\ell_{i}$ follows from $[G, \gamma_{n+m-1}]\subseteq \gamma_{n+m}$.
Thus we have
\begin{align*}
[h, k]^p  = h^p u\ell_{p-1}u\ell_{p-2}\dots, u\ell_1u \in (h^pu^p)\gamma_{n+m}^p(G),
\end{align*}
and in particular $[h^p, k]\gamma^p_{n+m}(G) =  (h^pu^p)\gamma_{n+m}^p(G) = [h, k]^p\gamma_{n+m}^p(G) $.
Since $[h,k]\in \gamma_{n+m-1}^p(G)$ by induction, we have $[h,k]^p\in\gamma_{n+m}^p(G)$ by Lemma~\ref{prop: lower p-central filtration basic}(1). 
We conclude that $[h^p,k]\in\gamma_{n+m}^p(G)$, as claimed.

Now we claim that $[[G,\gamma_{m-1}^p(G)], \gamma_{n}^p(G)]\subseteq \gamma_{n+m}^p(G)$.
By the three subgroup lemma (see e.g.\ \cite[Cor 8.28]{Isaacs2009})
\begin{align*}
[[G,\gamma_{m-1}^p(G)], \gamma_n^p(G)] 
&\subseteq [[\gamma_{m-1}^p(G), \gamma_n^p(G)],G]\cdot [[G, \gamma_n^p(G)], \gamma_{m-1}^p(G)]\\
&\subseteq [\gamma_{n+m-1}^p(G), G] \cdot [\gamma_{n+1}^p(G),  \gamma_{m-1}^p(G)]\\
&\subseteq \gamma_{n+m}^p(G)
\end{align*}
and the second and third line follow from the inductive hypothesis.
Thus we conclude that $[\gamma_{m}^p(G),\gamma_{n}^p(G)] = [(\gamma_{m-1}^p(G))^p[\gamma_{m-1}^p(G), G],\gamma_{n}^p(G)]\subseteq \gamma_{n+m}^p(G)$, as desired.

\item We now induct on $n$. For $n=1$, clearly $\gamma_1^p(G) = G$ is characteristic in $G$. We assume that the statements in true for $n-1$ and prove it for $n$. Let $h \in  \gamma_n^p(G) = \gamma_{n-1}^p(G)[G, \gamma_{n-1}^p(G)]$, i.e.\ $h = h_1^p\cdot [k,h_2]$ where $h_1, h_2\in \gamma_{n-1}^p(G)$ and $k\in G$. Let $\phi\in \Aut(G)$. Then
$$
\phi(h) = \phi(h_1^p\cdot [k,h_2]) = \phi(h_1)^p \cdot [\phi(k), \phi(h_2)].
$$
Since $\gamma_{n-1}^p(G)$ is characteristic, $\phi(h_1),\phi(h_2) \in \gamma_{n-1}^p(G)$, so $\phi(h) \in \gamma_{n-1}^p(G)[G, \gamma_{n-1}^p(G)] = \gamma_{n}^p(G)$. Thus $ \gamma_{n}^p(G)$ is characteristic.

Since $\gamma_{n}^p(G)$ is characteristic in $G$, every automorphism $\phi:G\to G$ preserves $\gamma_{n}^p(G)$, and therefore $\sigma_n(\phi):G/\gamma_{n}^p(G) \to G/\gamma_{n}^p(G)$ is well-defined. It is clear that $\sigma_n$ is a homomorphism. 
The automorphism $\phi$ restricts to $\phi_{|\gamma_{n}^p(G)}: \gamma_{n}^p(G)\to \gamma_{n}^p(G)$, and to $\phi_{|\gamma_{n}^p(G)}:\gamma_{n+1}^p(G)\to \gamma_{n+1}^p(G)$. Thus $\phi$ descends to a well-defined automorphism of $L_n^p(G)$. The map $\theta_n$ is clearly a homomorphism. 
\end{enumerate}
\end{proof}

\begin{Proposition}[{\cite[Chap VIII.1]{HuppertBlackburnBookII}}]
\label{prop: trivial action on L1}
Let $\phi \in \Aut(G)$ such that $\theta_1(\phi) = \Id_{L_{1}^p(G)}$. 
\begin{enumerate}
\item We have $\theta_n(\phi) = \Id_{L_{n}^p(G)}$ for all $n$. 
\item The order of $\sigma_n(\phi)$ is a power of $p$.
\end{enumerate} 
\end{Proposition}

\begin{proof}
\begin{enumerate}
\item
We induct on $n$. The case of $n=1$ is immediate. We assume that the statement holds for $n-1$.

Let first $h \in \gamma_{n-1}^p(G)$. Then by assumption $\phi(h) = hk_n$ where $k_n\in \gamma_n^p(G)$. We have
\begin{align*}
\phi(h^p \,\gamma_{n+1}^p(G)) &= (hk_n)^p \,\gamma_{n+1}^p(G)\\
& = h^p (h^{-(p-1)}k_nh^{p-1})(h^{-(p-2)}k_nh^{p-2})\cdots (h^{-1}k_nh)k_i \,\gamma_{n+1}^p(G) \\
&= h^p k_n\ell_{p-1}k_n\ell_{p-2}\cdots k_n\ell_1k_n \,\gamma_{n+1}^p(G)\\
&= h^pk_n^p (k_n^{-(p-1)}\ell_{p-1}k_n^{p-1})(k_n^{-(p-2)}\ell_{p-2}k_n^{p-2}) \dots (k_n^{-1}\ell_1k_n)\, \gamma_{n+1}^p(G)\\
&=h^p k_n^p \,\gamma_{n+1}^p(G) \\
&= h^p \,\gamma_{n+1}^p(G)
\end{align*}
where $\ell_1, \dots, \ell_{p-1}, \ell, \ell'$ are some elements of $\gamma_{n+1}^p(G)$. Indeed, the fact that $h^{-j}k_n h^{j} = k_n\ell_j$ follows from the fact that $[\gamma_{n-1}^p(G), \gamma_n^p(G)]\subseteq [G, \gamma_n^p(G)]\subseteq  \gamma_{n+1}^p(G)$ (Proposition~\ref{prop: lower p-central filtration basic}(2)). Finally, the fact that $k_n^p\in \gamma_{n+1}^p$ follows from Proposition~\ref{prop: lower p-central filtration basic}(1).

Now let $g\in G$. Then $\phi(g) = g\ell_2$ for some $\ell_2\in \gamma_2^p(G)$. We have
\begin{align*}
\phi([g,h]\,\gamma_{n+1}^p(G))  &= [g\ell_2, hk_n] \,\gamma_{n+1}^p(G)\\
&= g(\ell_2hk_n\inv{\ell_2})\inv{g}\inv{k_n}\inv{h}\,\gamma_{n+1}^p(G) \\
&= gh(k_{n}d_{n+1} \inv{g} \inv{k_n})\inv{h}\,\gamma_{n+1}^p(G)\\
&= gh d_{n+1}\inv{g}d_{n+1}' \inv{h}\,\gamma_{n+1}^p(G)\\
&= gh\inv{g}\inv{h}\,\gamma_{n+1}^p(G)\\
& = [g,h]\,\gamma_{n+1}^p(G)
\end{align*}
where we used Proposition~\ref{prop: lower p-central filtration basic}(2) to write $\ell_2hk_n\inv{\ell_2} = hk_nd_{n+1}$ for some $d_{n+1}\in \gamma_{n+1}^p(G)$ since $[\gamma_2^p(G), \gamma_{n-1}^p(G)]\subseteq \gamma_{n+1}^p(G)$, and $k_{n}d_{n+1} \inv{g} \inv{k_n} = d_{n+1}\inv{g}d_{n+1}'$ for some $d'_{n+1}\in \gamma_{n+1}^p(G)$ since $[\gamma_n^p(G), G]\subseteq \gamma_{n+1}^p(G)$.

Finally, every generator $h$ of $\gamma_n^p(G)$ is of the form $h = h_1^p [g, h_2]$ for some $h_1, h_2\in \gamma_{n-1}^p(G)$. We have
\begin{align*}
 \phi(h\gamma_{n+1}^p(G)) = \phi(h_1^p) \phi([g, h_2])\gamma_{n+1}^p(G) = h_1^p [g, h_2] \gamma_{n+1}^p(G).
 \end{align*}
 for some $\ell \in \gamma_{i+1}^p(G)$. This proves that $\theta_i(\phi) = \Id_{L_i^p(G)}$ as claimed.

\item We prove by induction on $n$ that $\sigma_n(\phi^{p^{n-1}}) = \Id_{G/\gamma^p_n(G))}$. The case of $n=1$ is immediate as $\sigma_1 =\theta_1$. We assume that the statement holds for $n-1$, in particular for every $h \in G$ we have $\phi^{p^{n-2}}(h\gamma_{n-1}^p(G)) = h\gamma_{n-1}^p(G)$, i.e.\ there exists $k\in \gamma_{n-1}^p(G)$ such that $\phi^{p^{n-2}}(h) = hk$. We have
\begin{align*}
\phi^{p^{n-1}}(h) = (\phi^{p^{n-2}})^{p}(h) = (\phi^{p^{n-2}})^{p-1}\left( \phi^{p^{n-2}}(h)\right)  
= (\phi^{p^{n-2}})^{p-1}(hk).  
\end{align*}

Since $k\in \gamma_{n-1}^p(G)$ and $\theta_n(\phi^{p^{n-2}}) = \theta_n(\phi)^{p^{n-2}} = \Id_{L_{n}^p(G)}$, using Proposition~\ref{prop: trivial action on L1}(1), we get
\begin{align*}
(\phi^{p^{n-2}})^{p-1}(hk) &= (\phi^{p^{n-2}})^{p-2}(\phi^{p^{n-2}}(h)\phi^{p^{n-2}}(k))\\ &\in (\phi^{p^{n-2}})^{p-2}(hk\cdot k\gamma_{n}^p(G))\\
 &= (\phi^{p^{n-2}})^{p-3}(hk^3\gamma_{n}^p(G)) =  \dots = hk^p\gamma_{n}^p(G).
\end{align*}
In particular, since $k^p\in \gamma_{n}^p(G)$ by Proposition~\ref{prop: lower p-central filtration basic}(1), we conclude $\phi^{p^{n-1}}(h) \in h\gamma_{n}^p(G)$ as required.
\end{enumerate}
\end{proof}

\begin{Corollary}\label{cor: trivial action of a subgroup}

 Let  $A\subseteq \Aut(G)$ be a subgroup.
\begin{enumerate}
\item The image $\theta_1(A) \subseteq \Aut(L_1^p(G))$ is a finite $p$-group if and only if  $\theta_n(A) \subseteq \Aut(L_n^p(G))$ is a finite $p$-group for every $n\geq 1$.
\item The image $\theta_1(A) \subseteq \Aut(G/\gamma_1^p(G))$ is a finite $p$-group if and only if  $\sigma_n(A) \subseteq \Aut(G/\gamma_n^p(G))$ is a finite $p$-group for every $n\geq 1$.
\end{enumerate}
\end{Corollary}

\begin{proof}
\item 
\begin{enumerate} 
\item
Fix $n\geq 1$, and suppose that $\theta_1(A)$ is a finite $p$-group. By Proposition~\ref{prop: trivial action on L1}(1) for every $\phi\in A$, if $\theta_1(\phi) = \Id_{L_1^p(G)}$, then  $\theta_n(\phi) = \Id_{L_i^p(G)}$. Thus $\theta_i(A)$ is a quotient of $\theta_1(A)$. In particular, $\theta_i(A)$ is a finite $p$-group as required.
\item Fix $n\geq 1$, and suppose that $\sigma_1(A) = \theta_1(A)$ is a finite $p$-group. For every $\phi\in A$ there exists $k\geq 1$ such that $\theta_1(\phi^{p^k}) = \theta_1(\phi)^{p^k}  =  \Id_{L_1^p(G)}$. By Proposition~\ref{prop: trivial action on L1}(2) the order of $\sigma_n(\phi^{p^k})$ is a power of $p$, and therefore the order of $\sigma_n(\phi)$ is a power of $p$. We conclude that $\sigma_n(A)$ is a $p$-group. As a subgroup of the automorphism group of a finite group $\sigma_n(A)$ is also finite.
\end{enumerate}
\end{proof}

The above corollary implies that for every subgroup $K$ of  $\ker(\Aut(G)\to \Aut(L_1^p(G))$ (which has finite index in $\Aut(G)$), all the images $\sigma_i(K)$ are finite $p$-groups. This observation is crucial in the proof of Theorem~\ref{thm: main residually p}.


We note the following observation that the operators $\theta_i$ and $\sigma_i$ can be extended to partial automorphisms.

\begin{Lemma}\label{lem: partial automorphisms descend}
A partial automorphism $\phi:N\to M$ induces a partial automorphism $N\cap \gamma_{n}^p(G)) \to M \cap \gamma_{n}^p(G)$ of $\gamma_n^p(G)$. In particular, it descends to the following partial automorphisms
\begin{enumerate}
\item $\sigma_n(\phi):N /N\cap \gamma_{n}^p(G) \to M/M\cap \gamma_{n}^p(G)$, and
\item $\theta_n(\phi):N\cap \gamma_{n}^p(G) /N\cap \gamma_{n+1}^p(G) \to M\cap \gamma_{n}^p(G)/M\cap \gamma_{n+1}^p(G)$.
\end{enumerate}
\end{Lemma}
\begin{proof}
Since $\phi$ is a partial automorphism of $G$, there exists $\phi^{ext}\in\Aut(G)$ such that $\phi^{ext}_{|N} = \phi$. 
By Proposition~\ref{prop: lower p-central filtration basic}(3), $\gamma_{n}^p(G)$ is a characteristic subgroup of $G$, so $\phi^{ext}(\gamma_{n}^p(G)) = \gamma_{n}^p(G)$ for every $n$. Thus $\phi(N\cap \gamma_{n}^p(G)) = \phi^{ext}(N\cap \gamma_{n}^p(G)) = M \cap \gamma_{n}^p(G)$ for every $n$. The lemma follows.
\end{proof}

\subsection{Residual $p$-finiteness criterion for graphs of groups}
We generalize a theorem of \cite{AschenbrennerFriedl13} that every graph of virtually residually $p$-groups, where edge group inclusions are isomorphisms, are virtually residually $p$-finite. Their result, in particular, applies to free-by-cyclic groups.
We will use a criterion for residual $p$-finiteness of graphs of groups stated therein. 

A \emph{filtration} ${\bf G}$ of a graph of groups $\mathcal G$ is a collection $\{{\bf G}_v\}_v$ of \emph{compatible} filtrations ${\bf G}_v = \{G_{v,n}\}_n$ of $G_v$ for each $v\in V(\Gamma)$, in the sense that for all $n$
$$
f_e^{-1}(G_{\tau(e),n}) = f_{\overline e}^{-1}(G_{\tau(\overline e), n}).
$$

 For a given property $X$ (e.g.\ normal, separating), we say that ${\bf G}$ is $X$ if for every $v\in V(Y)$ the filtration ${\bf G}_v$ is $X$. We say that a filtration  ${\bf G}$ of  a graph of groups $\mathcal G$ \emph{separates edge groups} if $f_e(G_e) = \bigcap_n G_{\tau(n),n}\cdot f_e(G_e)$ for all edges $e$.

 Let ${\bf G}_n$ be the $n$-th depth subgroups of the filtration ${\bf G}$ of $\mathcal G$, i.e.\ ${\bf G}_n = (G_{v,n})_{v\in V(Y)}$. Since the filtrations of the vertex groups are compatible, there exists a natural graph of groups quotient $\mathcal G/{\bf G}_n$ which has $Y$ as it underlying graph and vertex groups $G_{v}/G_{v,n}$.

The following will be used to prove Theorem~\ref{thm: main residually p}.
\begin{Theorem}[{\cite[Cor 3.14]{AschenbrennerFriedl13}}]
\label{thm: criterion AF}
Let ${\bf G}$ be a normal separating filtration of $\mathcal G$ which separates edge groups of $\mathcal G$, such that $\pi_1(\mathcal G/{\bf G}_n)$ is residually $p$-finite for every $n\geq 1$. Then $\pi_1 \mathcal G$ is residually $p$-finite.
\end{Theorem}

\subsection{Main proof}
The following lemma can be deduced from \cite{Wilkes2019} but in this case it is easy to prove it directly.

\begin{Lemma}\label{lem:graphs with partial identities}
Let $P$ be a finite $p$-group. Let $\mathcal G$ be a graph of $p$-groups where each vertex group comes with an injective homomorphism $\psi_v: G_v\to P$, and each edge group comes with an inclusion $\psi_e:G_e\to P$ such that $\psi_e = \psi_{\overline e}$. Moreover, assume that for each edge $e$ the composition $\psi_{\tau(e)}\cdot f_e = \psi_e$. Then $\pi_1\mathcal G$ is residually $p$-finite.
\end{Lemma}
\begin{proof}
The assumption on $\mathcal G$ imply that there exists an epimorphism $\psi: \pi_1\mathcal G \to P$ which is an isomorphism on each vertex group. Indeed, $\psi$ is defined as $\psi_v$ on each vertex group $G_v$, and sending all edge generators (not edge groups) to the identity.
The kernel $\ker \psi$ is thus a finite index subgroup of $\pi_1\mathcal G$ and splits as a finite graph of trivial groups, i.e.\  $\ker \psi$ is a finite rank free group. By Lemma~\ref{lem: res p}(2) $\pi_1\mathcal G$ is residually $p$-finite.
\end{proof}

We are now ready to prove the main theorem.

\begin{Theorem}\label{t:main}
For every prime $p$, the fundamental group of an algebraically clean graph of free groups has a finite index subgroup that is residually $p$-finite.
\end{Theorem}

To illustrate the proof, we first consider the special case of free-by-free groups.
\begin{proof}[Proof for $F\rtimes Q$] 
Let $F, Q$ be finite rank free groups. In particular, $F$ and $Q$ are residually $p$-finite. 
Let $\alpha: Q\to \Aut(F)$ be a homomorphism associated to the semi-direct product. 
By Lemma~\ref{prop: lower p-central filtration basic}(3) $\gamma_2^p(N)$ is characteristic in $N$, so every automorphism $\phi \in \Aut(F)$ descends to an automorphism of $L_1^p(F)$, i.e.\ there is a well-defined homomorphism $\beta:\Aut(F) \to \Aut(L_1^p(F))$. 
By composing $\alpha$ with $\beta$, we obtain a homomorphism to a finite group $\beta\cdot \phi: Q\to \Aut(L_1^p(F))$. Let $Q'$ be its kernel. 
Then $F\rtimes Q'$ is a finite index subgroup of $F\rtimes Q$, which we claim is residually $p$-finite.

By Corollary~\ref{cor: trivial action of a subgroup}(2) the image $\sigma_i(Q')$ is a $p$-group for every $i$. In particular $(F/\gamma_i^p)\rtimes Q'$ is residually $p$-finite. Indeed if $g\in (F/\gamma_i^p)\rtimes Q'$ survives in $Q'$, then we can use the fact that $Q'$ is residually $p$-finite. Otherwise, when $g \in \ker((F/\gamma_i^p)\rtimes Q'\to Q')$, then $g$ must survive in the quotient $(F/\gamma_i^p)\rtimes \sigma_i(Q')$, which is a $p$ group as its order is a power of $p$.
Since every element $g\in F\rtimes Q'$ survives in $F/\gamma_i^p(F)\rtimes Q'$ for some $i$, we conclude that $F\rtimes Q'$ is residually $p$-finite.
\end{proof}

We now move to the general case. Let $F$ be a finite rank free group, and let $N, M$ be two subgroups of the same rank, each being a free factor of $F$. Every isomorphism $\phi:N\to M$ can be extended to an automorphism $\phi^{ext}:F\to F$, i.e.\ $\phi^{ext}_{|N}  = \phi$, which we call an \emph{extension of $\phi$ to $F$}. Note that an extension of $\phi$ is  far from being unique. Indeed, it is only unique if $N = M = F$.

\begin{proof}[Proof in general case]
Let $G'$ be a fundamental group of an algebraically clean graph of free groups.
By Proposition~\ref{prop:multiple HNN} we can think of $G'$ as the fundamental group of a graph of groups $\mathcal G'$ where the underlying graph $Y'$ of $\mathcal G'$ is a wedge of $k$ oriented circles $\{e_1, \dots, e_n\}$. The unique vertex group of $\mathcal G'$ is identified with a finite rank free group $F$, and for each $1\leq i\leq k$ the edge groups $G_{e_i}$ can be identified with a free factor $N_i$ so that $f_{e_i} = \Id_{N_i}$ and $f_{\overline e_i} =\phi_i:N_i\to M_i$ is a partial automorphism onto a free factor $M_i$ of $F$.
Let $Q'$ be a free group freely generated by $\{\phi_1, \dots, \phi_k\}$, which can be naturally identified with $\pi_1(Y)$.

\begin{step} We construct a finite index normal subgroup $Q$ of $Q'$ such that $\theta_n(Q)$ is trivial.\end{step}

Let $ \{\phi_1^{ext}, \dots,\phi_k^{ext}\}$ be a choice of extensions of $\{\phi_1, \dots, \phi_k\}$, i.e.\ for each $1\leq i\leq k$ $\phi_i^{ext}\in \Aut(F)$ such that ${\phi_i^{ext}}_{|N_i} = \phi_i$.
Recall the homomorphism $\theta_1:\Aut(F) \to \Aut (L_1^p(F))$ from Proposition~\ref{prop: lower p-central filtration basic}(3).
We construct a subgroup $Q$ of $Q'$ as
$$ Q = \ker (Q'\to \Aut(F) \to \Aut(L_1^p(F))$$
where the first map sends $\phi_i$ to $\phi_i^{ext} \in \Aut(F)$, and the second map is $\theta_1$. Since $ \Aut(L_1^p(F))$ is a finite group, the index $[Q':Q]$ is finite.

By Corollary~\ref{cor: trivial action of a subgroup}(1), the image of $Q$ in $\Aut(L_n^p(F))$ is trivial for every $n$.

\begin{step} We construct the corresponding finite index normal subgroup $G$ of $G'$ and realize it as the fundamental group of a graph of groups $\mathcal G$ covering the graph of groups for $\mathcal G'$.\end{step}
Consider the finite index subgroup $G$ of $G'$ corresponding to $Q$, i.e.\ $G = \ker(G'\to Q' \to Q'/Q)$.
The group $G$ is the fundamental group of the following graph of groups $\mathcal G$. The underlying graph $Y$ of $\mathcal G$ is the finite covering space of $Y'$ corresponding to $Q\subseteq Q'$. Each vertex group of $\mathcal G$ is a copy of $F$. The edge groups of edges labelled with $e_i$ are copies of $N_i$ with the maps $\Id_{N_i}$ and $\phi_i$ into the respective vertex groups. We note that $\mathcal G$ is still an algebraically clean graph of finite rank free groups, with a natural quotient $Q$. 

\begin{step} There is a natural filtration ${\bf G}$ of $\mathcal G$ where ${\bf G}_v = \{\gamma_n^p(G_v)\}_n$ of $G_v$ for each $v\in V(Y)$. The filtration is normal, separating, and it separates the edge groups.\end{step}
Lemma~\ref{lem: partial automorphisms descend} implies that the filtrations ${\bf G}_v$ on the individual vertex groups are compatible and indeed define a filtration on $\mathcal G$. It is immediate that ${\bf G}$ is normal. Since all the vertex groups are free groups, hence residually $p$-finite, their lower $p$-central series are separating. 
Since the edge groups are retracts of the vertex groups, and vertex groups are residually $p$-finite, \cite[Lem 1.6]{AschenbrennerFriedl13} implies that the filtration separates the edge groups.


\begin{step} For each $n$, $\pi_1(\mathcal G/{\bf G}_n)$ is residually $p$-finite. \end{step}
The graph of groups $\mathcal G/{\bf G}_n$ has all the vertex groups naturally isomorphic to $F/\gamma_n^p(F)$ and the edge groups are $(N_i/(N_i\cap \gamma_{n}^p(F))$ for respective $i$, with the respective edge maps being partial identities and partial automorphisms $\phi_i$. 
We construct a further subgroup $Q_n$ of $Q$ as
$$ Q_n = \ker (Q\to \Aut(F) \to \Aut(F/\gamma_n^p(F))$$
the first map sends $\phi_i$ to $\phi_i^{ext} \in \Aut(F)$, and the second map is the map $\sigma_n$ defined in Lemma~\ref{prop: lower p-central filtration basic}(3). 
By Corollary~\ref{cor: trivial action of a subgroup}(2), the image of $Q$ in $\Aut(F/\gamma_n^p(F))$ is a finite $p$-group, and therefore $[Q:Q_n]$ is a $p$-power. 
We now claim that the kernel $K_n = \ker (\pi_1(\mathcal G/{\bf G}_n)\to Q\to Q/Q_n)$ is the fundamental group of a graph of groups satisfying the assumptions of Lemma~\ref{lem:graphs with partial identities}, and therefore is residually $p$-finite. 
Indeed, $K_n$ is a finite cover $\mathcal G_n$ of the graph of groups $\mathcal G/{\bf G}_n$ whose all the vertex groups are still naturally isomorphic to $F/\gamma_n^p(F)$, and edge groups are $(N_i/(N_i\cap \gamma_{n}^p(F))$ for respective $i$, with the respective edge maps being partial identities and partial automorphisms $\phi_i$. 
We fix a vertex group $G_{v_0}$ of $\mathcal G_n$ and for each $v\in V(Y)$ we construct a map $\phi_v:G_v\to G_{v_0}$.
We describe each map as an automorphism $\psi_v\in \Aut(F/\gamma_n^p(F))$ using the natural identification of each $G_v$ with $F/\gamma_n^p(F)$. 
First, $\psi_{v_0} = \Id_{F/\gamma_n^p(F)}$. For $v$ such that there is a path from $v$ to $v_0$ is labelled by edges $e_{i_1}\dots{e_{i_k}}$ we define $\psi_{v} = \sigma_n(\phi_{i_k}^{ext}\cdots \phi_{i_1}^{ext})$. 

We claim that $\psi_{v}$ does not depend on the choice of the path from $v$ to $v_0$.
Indeed, given some other path with labels  $e_{j_1}\dots{e_{j_{k'}}}$ we get that 
$$\theta_n(\phi_{i_k}^{ext}\cdots \phi_{i_1}^{ext})\inv{\theta_n(\phi_{j_{k'}}^{ext}\cdots \phi_{j_{1}}^{ext})}
 = \theta_n(\phi_{i_k}^{ext}\cdots \phi_{i_1}^{ext}\inv{(\phi_{j_1}^{ext})}\cdots \inv{(\phi_{j_{k'}}^{ext})}) = \Id_{G_{v_0}}
$$
by our choice of $Q_n$. This proves that $\theta_n(\phi_{i_k}^{ext}\cdots \phi_{i_1}^{ext}) = \theta_n(\phi_{j_{k'}}^{ext}\cdots \phi_{j_{1}}^{ext})$.
For any edge $e$ with label $i$, we set $\psi_e = \psi_{\overline e} = \psi_{\tau(e){|G_e}}$, and easily verify that this choice is compatible with both $\psi_{\tau(e)}$ and $\psi_{\tau(\overline e)}$. 
By Lemma~\ref{lem:graphs with partial identities} $\mathcal G_n$ is residually $p$-finite. 
By Lemma~\ref{lem: res p}, so is $\mathcal G/{\bf G}_n$ since it has a power-$p$ subgroup that is residually $p$-finite.

\begin{step} The group $\pi_1(\mathcal G)$ is residually $p$-finite.
\end{step}
Residual $p$-finiteness of $\pi_1(\mathcal G)$ follows from Theorem~\ref{thm: criterion AF} and Steps 3 and 4.
\end{proof}

\section{Cohomological $p$-completeness and goodness}

In this section we prove Theorem~\ref{thm: main residually p}. The proofs are nearly identical for goodness/$p$-completeness.  Therefore, we will just prove the $p$-completeness statements, and mention in the last subsection how the same arguments work for goodness. 
\subsection{Cohomological $p$-completeness}
Recall from the introduction that a discrete group $G$ is \emph{cohomologically $p$-complete} if the canonical homomorphism $G\to\widehat G_p$ to the pro-$p$ completion induces an isomorphism \[H^*_{cont}(\widehat G_p, \mathbb F_p)\to H^*(G, \mathbb F_p).\]

\begin{Theorem} The following groups are cohomologically $p$-complete for all $p$:

\begin{enumerate}
\item Free groups \cite{LinnelSchick07}
\item Finitely generated nilpotent groups \cite{lorensen08}
\item Right-angled Artin groups \cite{lorensen08}
\item Free products and direct products of cohomologically $p$-complete groups\cite{LinnelSchick07}. 
\item Retracts of cohomologically $p$-complete groups \cite{LinnelSchick07}. 

\end{enumerate}

\end{Theorem}

 We refer to \cite{LinnelSchick07} for further details on cohomologically $p$-complete groups. 
The idea behind our proof is simple. We are considering multiple HNN extensions of a free group $F$, and the cohomology of these can be computed by a Mayer-Vietoris sequence. 
It is well-known that free groups are cohomologically $p$-complete for every $p$, hence four out of every five terms in the Mayer-Vietoris sequence are cohomology groups of cohomologically $p$-complete groups, and the remaining term is $H^\ast(G, \mathbb{F}_p)$. If there was a similar exact sequence for the cohomology of the pro-$p$ completion $\widehat G^p$, then we would be done by the Five Lemma (this is essentially the argument for right-angled Artin groups used in \cite{lorensen08}). 
The following property of a graph of groups is a sufficient condition for this pro-$p$ Mayer-Vietoris sequence \cite[Lem 5.11]{AschenbrennerFriedl13}. A profinite version can be found in \cite[Prop 4.3]{WiltonZalesskii2010}.

\begin{Definition}
Let  $G$ be the fundamental group of a graph of groups, and suppose $G$ is residually $p$-finite. The pro-$p$ topology on $G$ is \emph{$p$-efficient} if the vertex and edge groups of G are closed in the pro-$p$ topology of $G$ and if the pro-$p$ topology on $G$ induces the full pro-$p$ topologies on the vertex and edge groups of $G$.
\end{Definition}

In general, if $H < G$ and $G$ is residually $p$-finite, the pro-$p$ topology on $G$ induces the full pro-$p$ topology on $H$ if and only if for every $p^n$-index subgroup $K < H$, there is a $p^m$-index subgroup $J<G$ with $J \cap H \subset K$. A subgroup $H<G$ is closed in the pro-$p$ topology if it is the intersection of $p^n$-index subgroups. 

\begin{Lemma}\label{l:retract}
Let $\mathcal G$ be a graph of groups where the edge groups are retracts of the vertex group. Then $\pi_1 \mathcal G$ is $p$-efficient if and only if 
\begin{enumerate}
\item $G = \pi_1 \mathcal G$ is residually $p$-finite,
\item the pro-$p$ topology on $G$ induces the pro-$p$ topology on $G_v$ for all vertices $v$, and
\item every vertex group $G_v$ is closed in the pro-$p$ topology of $G$.
\end{enumerate}
\end{Lemma}
\begin{proof} 
Every homomorphism $\phi:G_e\to P$ to a finite $p$-group $P$ extends to a homomorphism from $G_v$. This proves that the pro-$p$ topology on $G_v$ induces to the pro-$p$ topology on $G_e$, and it follows that the pro-$p$ topology on $G$ induces the pro-$p$ topology on $G_v$.
When the edge group $G_e$ is a retract of a vertex group $G_v$, then it is closed in the pro-$p$ topology of $G_v$ by \cite[Lem 1.6]{AschenbrennerFriedl13}. Thus if $G_v$ is closed in the pro-$p$ topology of $G$, then so is $G_e$. 
\end{proof}

We state the criterion.

\begin{Theorem}[{\cite[Cor 5.12]{AschenbrennerFriedl13}}]\label{thm: completeness AF}
Let $\mathcal G$ be a $p$-efficient graph of finitely generated groups, where all vertex and edge groups are cohomologically $p$-complete. Then $\pi_1\mathcal G$ is cohomologically $p$-complete.
\end{Theorem}

\begin{Theorem}
For every prime $p$, the fundamental group of an algebraically clean graph of free groups has a finite index subgroup which is cohomologically $p$-complete
\end{Theorem}

\begin{proof}
We have already constructed in the proof of Theorem \ref{t:main} a finite index subgroup $\pi_1(\mathcal G)$ of $\pi_1(\mathcal G')$ which is residually $p$-finite. We claim that the corresponding decomposition as a graph of free groups is efficient. 
Recall that the pro-$p$ topology on vertex groups is generated by the filtration $\gamma_n^p(G_v)$. 
By construction, for every $n$, $\pi_1(\mathcal G)$ admits a homomorphism to a finite $p$-group which restricts to $G_v \rightarrow G_n/\gamma_n^p(G_v)$. 
This combined with Lemma \ref{l:retract} shows that the pro-$p$ topology on $\pi_1(\mathcal G)$ is efficient, so we are done by Theorem \ref{thm: completeness AF}.
\end{proof}


\subsection{Goodness}

Cohomological goodness is a bit easier to establish; we will give a more straightforward proof without restating the relevant definitions (which essentially involves replacing the pro-$p$ completion/topology everywhere with the profinite completion/topology). 

\begin{Theorem}
The fundamental group $G$ of an algebraically clean graph of free groups is cohomologically good. 
\end{Theorem}

\begin{proof}
We know $G$ decomposes as an iterated HNN extension of a free group $F$, where the edge maps extend to automorphisms of $F$. 
We claim these decompositions are efficient. 
To see this, take a finite index characteristic subgroup $C$ of $F$. 
There is an induced homomorphism from $G$ to an iterated HNN extension of $F/C$, denoted by $G'$. 
Since $F/C$ is finite, $G'$ is virtually free, so let $H"$ be any finite index free subgroup which intersects $F/C$ trivially. 
The preimage $H$ of $H'$ hence intersects $F$ inside of $C$.
This shows the HNN extension is efficient, so we are done by the profinite Mayer-Vietoris sequence \cite[Prop 4.3]{WiltonZalesskii2010}.  
\end{proof}

Since finite extensions of good groups are good, this implies that any Artin group satisfying the conditions of Corollary \ref{cor:Artin} is good as well.

\section{Virtual poly-freeness}

\begin{proof}[Proof of Theorem~\ref{thm:poly-free}]
Let $\mathcal G$ be an algebraically clean graph of finite rank free groups over a finite graph $\Gamma$. By Proposition~\ref{prop:multiple HNN}, we can assume that $\Gamma$ has a unique vertex $v$, and some finite number of loops. Then the fundamental group of $\mathcal G$ fits in the following short exact sequence
\[
1\to \llangle G_v\rrangle \to \pi_1(\mathcal G)\to \pi_1\Gamma\to 1.
\]
We claim that $\llangle G_v\rrangle$ is a (possibly infinite rank) free group. Indeed, the induced graph of groups decomposition of $\llangle G_v\rrangle$ is an infinite tree of $G_v$, amalgamated along free factors. 
The chain $1 \unlhd \llangle G_v\rrangle \unlhd \pi_1(\mathcal G)$ is a chain witnessing the normal poly-freeness of $\pi_1(G(\Gamma))$.
\end{proof}

\section{Even Artin groups}

\begin{Lemma}\label{lem:evenartin-retractions}
Let $\Gamma$ be a graph labelled by even numbers $\geq 2$, and let $\Lambda \subseteq \Gamma$ be any induced subgraph. Then $A_{\Gamma}$ retracts onto $A_{\Lambda}$.
\end{Lemma}
\begin{proof} The retraction is obtained by mapping each generator $s\in V(\Lambda)$ to itself, and each generator $s\in V(\Gamma)-V(\Lambda)$ to $1$.
\end{proof}

\begin{proof}[Proof of Theorem~\ref{thm:evenArtin}]
The proof is an induction on the number of non-edges in the defining graph $\Gamma$ of $A_{\Gamma}$. If $\Gamma$ is a full graph, then by the assumption on no $4$-cliques $\Gamma$ has at most three vertices. 
If $\Gamma$ has one vertex, then $A_{\Gamma}=\mathbb Z$ is residually finite and cohomologically good. 
If $\Gamma$ has two vertices, then $A_{\Gamma}$ is virtually $F\times\mathbb Z$ (see e.g.\ \cite[Lem 4.3]{HuangJankiewiczPrzytycki16}), which is residually finite and cohomologically good. Since residual finiteness and cohomological goodness pass to finite index supergroups \cite[Lem 3.2]{GJ-ZZ08}, $A_{\Gamma}$ has those properties. Finally, when $\Gamma$ has three vertices then $A_{\Gamma}$ is residually finite and cohomologically good by Corollary~\ref{cor:Artin}.

We now prove the inductive step. Let $\Gamma$ has a non-edge $\{u,v\}$. 
Then $A_{\Gamma}$ splits as an amalgamated product $A_{\Gamma-\{u\}}*_{A_{\Gamma-\{u,v\}}} A_{\Gamma-\{v\}}$.
By the inductive assumption $A_{\Gamma-\{u\}}$ and $A_{\Gamma-\{v\}}$ are residually finite and cohomologically good.
By Lemma~\ref{lem:evenartin-retractions}, $A_{\Gamma}$ is an amalgamated product along retracts, so $A_{\Gamma}$ is residually finite by \cite{BolerEvans73}, and cohomologically good by \cite[Prop 3.5]{GJ-ZZ08}.
\end{proof}

More generally, the argument in the proof above shows that for even Artin groups being residually finite and cohomologically good are ``free-of-infinity properties'', in the sense that if we prove that all even Artin groups whose defining graphs are cliques are residually finite and cohomologically good, then we will be able to deduce that the same holds for all even Artin groups.

\bibliographystyle{alpha}
\bibliography{../../kasia}

\end{document}